\documentclass[11pt,a4paper]{article}
\usepackage[latin1]{inputenc}
\usepackage{amsmath}
\usepackage{amsthm}
\usepackage{amsfonts}
\usepackage{amsfonts,amsthm,latexsym,amsmath,amssymb,amscd,epsfig,psfrag,enumerate}
\usepackage{graphics,graphicx, bezier, float, color, hyperref}
\usepackage{amssymb,url}
\usepackage{multienum}
\usepackage[table]{xcolor}
\usepackage{multicol,multirow}
\usepackage{graphicx}
\usepackage{fancyvrb}
\usepackage{parskip}
\usepackage[toc,page]{appendix}
\usepackage[top=2.7cm, bottom=2.7cm, left=1.5cm, right=1.5cm]{geometry}
\usepackage{xcolor}

\usepackage{blkarray}
\newtheorem{theorem}{Theorem}[section]
\newtheorem{lemma}{Lemma}[section]

\usepackage[none]{hyphenat}[section]
\newtheorem{definition}{Definition}[section]

\numberwithin{equation}{section}
\numberwithin{table}{section}
\numberwithin{figure}{section}

\title{On the Largest Prime factor of the $k$-generalized Pell numbers}
\author{Herbert Batte$^{1,*} $}
\date{}

\begin{document}
\maketitle
\abstract{ Let $k \ge 2$ be an integer and consider the $k$-generalized Pell sequence $\{P_n^{(k)}\}_{n \ge 2-k}$, defined by the initial values $0, \ldots, 0, 0, 1$ (a total of $k$ terms), and the recurrence $P_n^{(k)} = 2P_{n-1}^{(k)} + P_{n-2}^{(k)} + \cdots + P_{n-k}^{(k)}$, for all $n\ge 2$.	For any integer $m$, let $\mathcal{P}(m)$ denote the largest prime factor of $m$, with the convention $\mathcal{P}(0) = \mathcal{P}(\pm1) = 1$. In this paper, we prove that for $n \ge 4$, the inequality $\mathcal{P}(P_n^{(k)}) > (1/104) \log \log n$ holds. Additionally, we find all $k$-generalized Pell numbers $P_n^{(k)}$, whose largest prime factor does not exceed $7$.
 } 



\thanks{$ ^{*} $ Corresponding author}

\section{Introduction}
\subsection{Background}
Let $k \ge 2$ be a fixed integer. The sequence of $k$-generalized Pell numbers, denoted by $\{P_n^{(k)}\}$, is defined recursively by
\[
P_n^{(k)} = 2P_{n-1}^{(k)} + P_{n-2}^{(k)} + \cdots + P_{n-k}^{(k)}, \qquad \text{for all }\qquad n \ge 2,
\]
with initial terms given by $P_{2-k}^{(k)} = \cdots = P_{-1}^{(k)} = P_0^{(k)} = 0$ and $P_1^{(k)} = 1$. When $k = 2$, the sequence coincides with the classical Pell numbers, and in that case we drop the superscript ${}^{(k)}$ for simplicity. The first few values for $P_n^{(k)}$, for some small values of $k$, are given in Table \ref{tab} below.
\begin{table}[h!]\label{tab}
	\centering
	\caption{First non-zero $k$-Pell numbers}
	\begin{tabular}{|c|c|l|}
		\hline
		$k$ & Name & First non-zero terms \\
		\hline \hline 
		2 & Pell   & 1, 2, 5, 12, 29, 70, 169, 408, 985, 2378, 5741, 13860, 33461, $\ldots$ \\
		3 & 3-Pell & 1, 2, 5, 13, 33, 84, 214, 545, 1388, 3535, 9003, 22929, 58396, $\ldots$ \\
		4 & 4-Pell & 1, 2, 5, 13, 34, 88, 228, 591, 1532, 3971, 10293, 26680, 69156, $\ldots$ \\
		5 & 5-Pell & 1, 2, 5, 13, 34, 89, 232, 605, 1578, 4116, 10736, 28003, 73041, $\ldots$ \\
		6 & 6-Pell & 1, 2, 5, 13, 34, 89, 233, 609, 1592, 4162, 10881, 28447, 74371, $\ldots$ \\
		7 & 7-Pell & 1, 2, 5, 13, 34, 89, 233, 610, 1596, 4176, 10927, 28592, 74815, $\ldots$ \\
		8 & 8-Pell & 1, 2, 5, 13, 34, 89, 233, 610, 1597, 4180, 10941, 28638, 74960, $\ldots$ \\
		9 & 9-Pell & 1, 2, 5, 13, 34, 89, 233, 610, 1597, 4181, 10945, 28652, 75006, $\ldots$ \\
		10 & 10-Pell & 1, 2, 5, 13, 34, 89, 233, 610, 1597, 4181, 10946, 28656, 75020, $\ldots$ \\
		\hline 
	\end{tabular}
\end{table}

For any integer $m$, we denote by $\mathcal{P}(m)$ the largest prime divisor of $m$, adopting the convention that $\mathcal{P}(0) = \mathcal{P}(\pm 1) = 1$. Investigating lower bounds for the largest prime factor in linear recurrence sequences is a problem that has attracted considerable attention in Number theory. Some important contributions have been made in this area (see, for instance, \cite{Batte} and \cite{Brl}). In this paper, we examine the properties of the $k$-generalized Pell numbers with the objective of establishing explicit lower bounds for $\mathcal{P}(P_n^{(k)})$ in terms of the parameters $k$ and $n$. Note that since $P_1^{(k)}=1$, $P_2^{(k)}=2$ and $P_3^{(k)}=5$ for all $k\ge 2$, we have that $\mathcal{P}(P_n^{(k)})$ is either 1, 2 or 5, for all $k\ge 2$ and $n\le 3$.  Therefore, we work with $n\ge 4$.

We prove the following results.
\subsection{Main Results}
\begin{theorem}\label{1.1p} 
	Let $\{P_n^{(k)}\}_{n \ge 2-k}$ be the sequence of $k$-generalized Pell numbers. Then, the inequality 
	\begin{align*}
	\mathcal{P}(P_n^{(k)})>\dfrac{1}{104}\log \log n,
	\end{align*}
holds for all $n\ge 4$ and $k\ge 2$.
\end{theorem}

\begin{theorem}\label{1.2p} 
	The only solutions to the Diophantine equation 
	\begin{align}\label{eq1.1p}
	P_n^{(k)}=2^a\cdot 3^b\cdot 5^c\cdot 7^d,
	\end{align}
	 in non-negative integers $n$, $k$, $a$, $b$, $c$, $d$, $e$ with $k \ge 2$ and $n \ge 4$, are 
	 \begin{align*}
	 &P_4^{(2)}=12,\qquad P_6^{(2)}=70,\qquad P_6^{(3)}=84,\qquad\text{and}\qquad P_{10}^{(5)}=4116.
	 \end{align*} 
\end{theorem}

Our approach can be summarized as follows. 
\begin{enumerate}[(a)]
	\item We utilize lower bounds for linear forms in logarithms of algebraic numbers to derive an upper estimate for $\log n$ in terms of the parameters $k$ and $\mathcal{P}(P_n^{(k)})$. The analysis is more straightforward when $k$ is small. For larger values of $k$, we take advantage of the fact that the dominant root of the $k$-generalized Pell sequence is very close to $\varphi^2$, where $$\varphi := \dfrac{1+\sqrt{5}}{2},$$
	is the golden ratio. This allows us to simplify our computations by approximating this root with $\varphi$ in the relevant logarithmic expressions. 
	\item Moreover, the techniques described here, combined with the use of the LLL-algorithm, are applied in the concluding part of the paper to determine all $k$-generalized Pell numbers whose largest prime factor is at most $7$.
\end{enumerate}

\section{Methods}
\subsection{Some identities on $k$-Pell numbers}
We begin with an identity from \cite{kilis}, which relates $P_n^{(k)}$ with the $m$-th Fibonacci number. It states that
\begin{align*}\label{pnfn}
	P_n^{(k)} =  F_{2n-1},
\end{align*}
for all $1 \le n \le k+1$, while the next term is $P_{k+2}^{(k)} =  F_{2k+3}-1$. The characteristic polynomial of $\{P_n^{(k)}\}_{n\in {\mathbb Z}}$ is given by
\[
\Psi_k(x) = x^k - 2x^{k-1} - \cdots - x - 1.
\]
This polynomial is irreducible in $\mathbb{Q}[x]$ and it possesses a unique real root $\alpha:=\alpha(k)>1$. All the other roots of $\Psi_k(x)$ are inside the unit circle,  see \cite{Herera}. The particular root $\alpha$ can be found in the interval
\begin{align}\label{eq2.3m}
	\varphi^2(1 - \varphi^{-k} ) < \alpha < \varphi^2,\qquad \text{for} \qquad k\ge 2,
\end{align}
as noted in \cite{Herera}. As in the classical case when $k=2$, it was shown in \cite{Herera} that 
\begin{align}\label{pn_bound}
	\alpha^{n-2} \le P_n^{(k)}\le \alpha^{n-1}, \qquad \text{holds for all}\qquad n\ge1 \quad\text{and}\quad k\ge 2.
\end{align}

Next, for $k\ge 2$, we define
\begin{align}\label{eq:fk}
	f_k(x):=\dfrac{x - 1}{(k + 1)x^2 - 3kx + k - 1} = \dfrac{x - 1}{k(x^2 - 3x + 1) + x^2 - 1}.
\end{align}
In Lemma 1 of \cite{BraHer}, Bravo and Herera showed that
\begin{align}\label{fk}
	0.276<f_k(\alpha)<0.5 \qquad\text{and}\qquad |f_k(\alpha_i)|<1,
\end{align}
holds for all $2\le i\le k$, where $\alpha_i$ for $i=2,\ldots,k$ are the roots of $\Psi_k(x)$ inside the unit circle. So, the number $f_k(\alpha)$ is not an algebraic integer. Indeed, if we suppose that $f_k(\alpha)$ is an algebraic integer, then we know that $f_k(\alpha) \neq 0$ and 
$N_{\mathbb{Q}(\alpha)/\mathbb{Q}}(f_k(\alpha)) \in \mathbb{Z}$, where $N_{\mathbb{Q}(\alpha)/\mathbb{Q}}(f_k(\alpha))$ is the norm of $f_k(\alpha)$ from $\mathbb{Q}(\alpha)$ to $\mathbb{Q}$. Therefore, we obtain
\[
1 \leq \left| N_{\mathbb{Q}(\alpha)/\mathbb{Q}}(f_k(\alpha)) \right| = f_k(\alpha) \cdot \prod_{i=2}^{k} \left| f_k(\alpha^{(i)}) \right| < 0.5,
\]
which is absurd. So $f_k(\alpha)$ is not an algebraic integer. Moreover, it was also shown in Theorem 3.1 of \cite{Herera} that
\begin{align}\label{eq2.6m}
	P_n^{(k)}=\displaystyle\sum_{i=1}^{k}f_k(\alpha_i)\alpha_i^{n}\qquad\text{and}\qquad\left|P_n^{(k)}-f_k(\alpha)\alpha^{n}\right|<\dfrac{1}{2},
\end{align}
hold for all $k\ge 2$ and $n\ge 2-k$. The first part of \eqref{eq2.6m} provides a Binet-type representation for \( P_n^{(k)} \). In addition, the second inequality in \eqref{eq2.6m} indicates that the impact of the roots located inside the unit circle on \( P_n^{(k)} \) is small.

Next, let $n\ge 4$ and 
\begin{align}\label{pf}
P_n^{(k)}=p_1^{\beta_1}\cdot p_2^{\beta_2}\cdots p_s^{\beta_s},	
\end{align}
be the prime factorization of the positive integer $P_n^{(k)}$, where $2=p_1<p_2<\cdots <p_s$ is the increasing sequence of prime numbers and the integers $\beta_i$, for $i=1,2,\ldots, s$, are non-negative integers. By the right hand side of relation \eqref{pn_bound}, we have 
\[
\log P_n^{(k)} \le (n-1)\log\alpha < 2(n-1)\log\varphi,
\]
since $\alpha<\varphi^2$, for all $k\ge 2$ in \eqref{eq2.3m}. Therefore, we can write
\begin{align*}
\sum_{i=1}^{s}\beta_i\log p_i=	\log\prod_{i=1}^s p_i^{\beta_i} &  =\log P_n^{(k)}< 2(n-1)\log\varphi,
\end{align*}	
so that
\begin{align*}
\sum_{i=1}^{s}\beta_i\log 2 & \le \sum_{i=1}^{s}\beta_i\log p_i <2(n-1)\log\varphi, \quad \text{since}\quad 2\le p_i, \quad \text{for all}\quad i=1,2,\ldots s.
\end{align*}
Dividing both sides by $\log 2$, we get
$$
\sum_{i=1}^{s}\beta_i <1.4(n-1)<1.4n.
$$
In particular, 
\begin{align}\label{eq2.7p}
	\beta_i <1.4n, 
\end{align}
for all $i=1,2,\ldots s$.

\subsection{Linear forms in logarithms}
To establish our result, we apply a lower bound of Baker type for non-zero linear combinations of logarithms of algebraic numbers. Although various versions of such bounds exist in the literature, we make use of a result due to Matveev as presented in \cite{MAT}. Before stating the relevant inequality, we first recall the definition of the height of an algebraic number.
\begin{definition}\label{def2.1}
	Let $ \gamma $ be an algebraic number of degree $ d $ with minimal primitive polynomial over the integers $$ a_{0}x^{d}+a_{1}x^{d-1}+\cdots+a_{d}=a_{0}\prod_{i=1}^{d}(x-\gamma^{(i)}), $$ where the leading coefficient $ a_{0} $ is positive. Then, the logarithmic height of $ \gamma$ is given by $$ h(\gamma):= \dfrac{1}{d}\Big(\log a_{0}+\sum_{i=1}^{d}\log \max\{|\gamma^{(i)}|,1\} \Big). $$
\end{definition}
In particular, if $ \gamma$ is a rational number represented as $\gamma=p/q$ with coprime integers $p$ and $ q\ge 1$, then $ h(\gamma ) = \log \max\{|p|, q\} $. 
The following properties of the logarithmic height function $ h(\cdot) $ will be used in the rest of the paper without further reference:
\begin{equation}\nonumber
	\begin{aligned}
		h(\gamma_{1}\pm\gamma_{2}) &\leq h(\gamma_{1})+h(\gamma_{2})+\log 2;\\
		h(\gamma_{1}\gamma_{2}^{\pm 1} ) &\leq h(\gamma_{1})+h(\gamma_{2});\\
		h(\gamma^{s}) &= |s|h(\gamma)  \quad {\text{\rm valid for}}\quad s\in \mathbb{Z}.
	\end{aligned}
\end{equation}
With these properties, it was proved in \cite[Lemma 3]{BraHer} that 
\begin{align}\label{eqh}
	h\left(f_k(\alpha)\right)<4k\log \varphi+k\log (k+1), \qquad \text{for all}\qquad k\ge 2.
\end{align}

A linear form in logarithms is an expression
\begin{equation}
	\label{eq:Lambdap}
	\Lambda:=b_1\log \gamma_1+\cdots+b_t\log \gamma_t,
\end{equation}
where for us $\gamma_1,\ldots,\gamma_t$ are positive real  algebraic numbers and $b_1,\ldots,b_t$ are integers. We assume, $\Lambda\ne 0$. We need lower bounds 
for $|\Lambda|$. We write ${\mathbb K}:={\mathbb Q}(\gamma_1,\ldots,\gamma_t)$ and $D$ for the degree of ${\mathbb K}$ over ${\mathbb Q}$.
We give Matveev's inequality from \cite{MAT}. 

\begin{theorem}[Matveev, \cite{MAT}]
	\label{thm:Matp} 
	Put $\Gamma:=\gamma_1^{b_1}\cdots \gamma_t^{b_t}-1=e^{\Lambda}-1$. Then 
	$$
	\log |\Gamma|>-1.4\cdot 30^{t+3}\cdot t^{4.5} \cdot D^2 (1+\log D)(1+\log B)A_1\cdots A_t,
	$$
	where $B\ge \max\{|b_1|,\ldots,|b_t|\}$ and $A_i\ge \max\{Dh(\gamma_i),|\log \gamma_i|,0.16\}$ for $i=1,\ldots,t$.
\end{theorem}

During the calculations, upper bounds on the variables are obtained which are too large, thus there is need to reduce them. To do so, we use some results from
approximation lattices and the so-called LLL-reduction method from \cite{LLL}. We explain this in the following subsection.

\subsection{Reduced Bases for Lattices and LLL-reduction methods}\label{sec2.3}
Let $k$ be a positive integer. A subset $\mathcal{L}$ of the $k$-dimensional real vector space ${ \mathbb{R}^k}$ is called a lattice if there exists a basis $\{b_1, b_2, \ldots, b_k \}$ of $\mathbb{R}^k$ such that
\begin{align*}
	\mathcal{L} = \sum_{i=1}^{k} \mathbb{Z} b_i = \left\{ \sum_{i=1}^{k} r_i b_i \mid r_i \in \mathbb{Z} \right\}.
\end{align*}
We say that $b_1, b_2, \ldots, b_k$ form a basis for $\mathcal{L}$, or that they span $\mathcal{L}$. We
call $k$ the rank of $ \mathcal{L}$. The determinant $\text{det}(\mathcal{L})$, of $\mathcal{L}$ is defined by
\begin{align*}
	\text{det}(\mathcal{L}) = | \det(b_1, b_2, \ldots, b_k) |,
\end{align*}
with the $b_i$'s being written as column vectors. This is a positive real number that does not depend on the choice of the basis (see \cite{Cas}, Section 1.2p).

Given linearly independent vectors $b_1, b_2, \ldots, b_k$ in $ \mathbb{R}^k$, we refer back to the Gram-Schmidt orthogonalization technique. This method allows us to inductively define vectors $b^*_i$ (with $1 \leq i \leq k$) and real coefficients $\mu_{i,j}$ (for $1 \leq j \leq i \leq k$). Specifically,
\begin{align*}
	b^*_i &= b_i - \sum_{j=1}^{i-1} \mu_{i,j} b^*_j,~~~
	\mu_{i,j} = \dfrac{\langle b_i, b^*_j\rangle }{\langle b^*_j, b^*_j\rangle},
\end{align*}
where \( \langle \cdot , \cdot \rangle \)  denotes the ordinary inner product on \( \mathbb{R}^k \). Notice that \( b^*_i \) is the orthogonal projection of \( b_i \) on the orthogonal complement of the span of \( b_1, \ldots, b_{i-1} \), and that \( \mathbb{R}b_i \) is orthogonal to the span of \( b^*_1, \ldots, b^*_{i-1} \) for \( 1 \leq i \leq k \). It follows that \( b^*_1, b^*_2, \ldots, b^*_k \) is an orthogonal basis of \( \mathbb{R}^k \). 
\begin{definition}
	The basis $b_1, b_2, \ldots, b_n$ for the lattice $\mathcal{L}$ is called reduced if
	\begin{align*}
		\| \mu_{i,j} \| &\leq \frac{1}{2}, \quad \text{for} \quad 1 \leq j < i \leq n,~~
		\text{and}\\
		\|b^*_{i}+\mu_{i,i-1} b^*_{i-1}\|^2 &\geq \frac{3}{4}\|b^*_{i-1}\|^2, \quad \text{for} \quad 1 < i \leq n,
	\end{align*}
	where $ \| \cdot \| $ denotes the ordinary Euclidean length. The constant $ {3}/{4}$ above is arbitrarily chosen, and may be replaced by any fixed real number $ y $ in the interval ${1}/{4} < y < 1$ (see \cite{LLL}, Section 1).
\end{definition}
Let $\mathcal{L}\subseteq\mathbb{R}^k$ be a $k-$dimensional lattice  with reduced basis $b_1,\ldots,b_k$ and denote by $B$ the matrix with columns $b_1,\ldots,b_k$. 
We define
\[
l\left( \mathcal{L},y\right)= \left\{ \begin{array}{c}
	\min_{x\in \mathcal{L}}||x-y|| \quad  ;~~ y\not\in \mathcal{L}\\
	\min_{0\ne x\in \mathcal{L}}||x|| \quad  ;~~ y\in \mathcal{L}
\end{array}
\right.,
\]
where $||\cdot||$ denotes the Euclidean norm on $\mathbb{R}^k$. It is well known that, by applying the
LLL-algorithm, it is possible to give in polynomial time a lower bound for $l\left( \mathcal{L},y\right)$, namely a positive constant $c_1$ such that $l\left(\mathcal{L},y\right)\ge c_1$ holds (see \cite{SMA}, Section V.4).
\begin{lemma}\label{lem2.5p}
	Let $y\in\mathbb{R}^k$ and $z=B^{-1}y$ with $z=(z_1,\ldots,z_k)^T$. Furthermore, 
	\begin{enumerate}[(i)]
		\item if $y\not \in \mathcal{L}$, let $i_0$ be the largest index such that $z_{i_0}\ne 0$ and put $\sigma:=\{z_{i_0}\}$, where $\{\cdot\}$ denotes the distance to the nearest integer.
		\item if $y\in \mathcal{L}$, put $\sigma:=1$.
	\end{enumerate}
	\noindent Finally, set 
	\[
	c_1:=\max\limits_{1\le j\le k}\left\{\dfrac{||b_1||}{||b_j^*||}\right\}\qquad\text{and}\qquad
	c_2:= c_1^{-1}\sigma||b_1||.
	\]
\end{lemma}
In our application, we are given real numbers $\eta_0,\eta_1,\ldots,\eta_k$ which are linearly independent over $\mathbb{Q}$ and two positive constants $c_3$ and $c_4$ such that 
\begin{align}\label{2.9pp}
	|\eta_0+x_1\eta_1+\cdots +x_k \eta_k|\le c_3 \exp(-c_4 H),
\end{align}
where the integers $x_i$ are bounded as $|x_i|\le X_i$ with $X_i$ given upper bounds for $1\le i\le k$. We write $X_0:=\max\limits_{1\le i\le k}\{X_i\}$. The basic idea in such a situation, due to \cite{Weg}, is to approximate the linear form \eqref{2.9pp} by an approximation lattice. So, we consider the lattice $\mathcal{L}$ generated by the columns of the matrix
$$ \mathcal{A}=\begin{pmatrix}
	1 & 0 &\ldots& 0 & 0 \\
	0 & 1 &\ldots& 0 & 0 \\
	\vdots & \vdots &\vdots& \vdots & \vdots \\
	0 & 0 &\ldots& 1 & 0 \\
	\lfloor C\eta_1\rfloor & \lfloor C\eta_2\rfloor&\ldots & \lfloor C\eta_{k-1}\rfloor& \lfloor C\eta_{k} \rfloor
\end{pmatrix} ,$$
where $C$ is a large constant usually of the size of about $X_0^k$ . Let us assume that we have an LLL-reduced basis $b_1,\ldots, b_k$ of $\mathcal{L}$ and that we have a lower bound $l\left(\mathcal{L},y\right)\ge c_1$ with $y:=(0,0,\ldots,-\lfloor C\eta_0\rfloor)$. Note that $ c_2$ can be computed by using the results of Lemma \ref{lem2.5p}. Then, with these notations the following result  is Lemma VI.1 in \cite{SMA}.
\begin{lemma}[Lemma VI.1 in \cite{SMA}]\label{lem2.6p}
	Let $S:=\displaystyle\sum_{i=1}^{k-1}X_i^2$ and $T:=\dfrac{1+\sum_{i=1}^{k}X_i}{2}$. If $c_2^2\ge T^2+S$, then inequality \eqref{2.9pp} implies that we either have $x_1=x_2=\cdots=x_{k-1}=0$ and $x_k=-\dfrac{\lfloor C\eta_0 \rfloor}{\lfloor C\eta_k \rfloor}$, or
	\[
	H\le \dfrac{1}{c_4}\left(\log(Cc_3)-\log\left(\sqrt{c_2^2-S}-T\right)\right).
	\]
\end{lemma}
Finally, we present an analytic argument which is Lemma 7 in \cite{GL}.  
\begin{lemma}[Lemma 7 in \cite{GL}]\label{Guzp} If $ m \geq 1 $, $T > (4m^2)^m$ and $T > \displaystyle \frac{x}{(\log x)^m}$, then $$x < 2^m T (\log T)^m.$$	
\end{lemma}
Python is used to perform all computations in this work.

\section{Proof of Theorem \ref{1.1p}.}\label{Sec3p}
In this section, we prove Theorem \ref{1.1p}. To do this, we first obtain bounds on $n$ in terms of $s$ and $k$. 
\subsection{An upper bound on $n$ in terms of $s$ and $k$.}\label{subsec3.1p}
From our earlier deduction that $\mathcal{P}(P_n^{(k)})$ is either 1, 2 or 5, for all $k\ge 2$ and $n\le 3$, we can  assume that $n\ge 4$ and $k\ge 2$. Moreover, in an investigation to determine the $k$-Pell numbers which are close to a power of 2, it was shown in \cite{Bach} that the the only powers of 2 in the $k$-Pell sequence are $P_1^{(k)}=1$ and $P_2^{(k)}=2$. So we may also assume $s\ge 2$. We prove the following result.
\begin{lemma}\label{lem3.1p}
Let $n \geq 4$, $k\ge 2$, $s\ge 2$ and $P_n^{(k)}$ have the prime factorization in \eqref{pf}. Then
	\[ \log n < 27s \log s + 5s \log k +  \log(10s+2k). \]
	
\end{lemma}
\begin{proof}
By the prime factorization of $P_n^{(k)}$ in \eqref{pf} and the relation in \eqref{eq2.6m}, we have 
\begin{align}\label{eq3.1p}
	\left|p_1^{\beta_1} \cdots p_s^{\beta_s}-f_k(\alpha)\alpha^{n}\right|<\dfrac{1}{2}.
\end{align}
Dividing both sides by $f_k(\alpha)\alpha^{n}$, which is positive because $\alpha>1$, we get
\begin{align}\label{eq3.2p}
	\left|p_1^{\beta_1} \cdots p_s^{\beta_s}\cdot\alpha^{-n}\cdot(f_k(\alpha))^{-1}-1\right|&<\dfrac{1}{2f_k(\alpha)\alpha^{n}}
	<\dfrac{1.82}{\alpha^{n}},
\end{align}
where in the second inequality, we used the fact that $f_k(\alpha)>0.276$ from \eqref{fk}. Let 
$$
\Gamma_1:=p_1^{\beta_1} \cdots p_s^{\beta_s}\cdot\alpha^{-n}\cdot(f_k(\alpha))^{-1}-1=e^{\Lambda_1}-1.
$$
Clearly, $\Gamma_1\ne 0$, otherwise we would have
\begin{align*}
f_k(\alpha) = p_1^{\beta_1} \cdots p_s^{\beta_s}\cdot\alpha^{-n}.
\end{align*}
The above equation shows that $f_k(\alpha)$ is an algebraic integer, which is not true. So, $\Gamma_1\ne 0$. 

The algebraic number field containing the following $\gamma_i$'s is $\mathbb{K} := \mathbb{Q}(\alpha)$. We have $D = k$, $t := s+2$,
\begin{equation}\nonumber
	\begin{aligned}
		\gamma_{i}&:=p_i~~ \text{for}~i=1,2,\ldots, s, ~~~\gamma_{s+1}:=\alpha,~~~\gamma_{s+2}:=f_k(\alpha),\\
		b_{i}&:=\beta_i~~ \text{for}~i=1,2,\ldots, s, ~~~b_{s+1}:=-n, ~~b_{s+2}:=-1.
	\end{aligned}
\end{equation}
Since $h(\gamma_{i})=\log p_i \le \log p_s$ for all $i=1,2,\ldots, s$, we take $A_i:=k \log p_s$ for all $i=1,2,\ldots, s$.  Additionally, $h(\gamma_{s+1})=h(\alpha) < (\log \alpha)/k<(2\log \varphi)/k < 1/k $, so we take $A_{s+1}:=1$. Lastly, 
\[
h(\gamma_{s+2}) = h\left(f_k(\alpha)\right) <  4k \log \varphi + k \log(k+1)  < 4.5k \log k,
\] 
by relation \eqref{eqh}. Hence, we take $A_{s+2}:=4.5k^2\log k$.

Next, $B \geq \max\{|b_i|:i=1,2,\ldots, s,\ldots, s+2\}$. Notice that $b_i=\beta_i<1.4n$, for all $i=1,2,\ldots, s$ by relation  \eqref{eq2.7p}, so we take $B:=1.4n$. Now, by Theorem \ref{thm:Matp},
\begin{align}\label{eq3.4p}
	\log |\Gamma_1| &> -1.4\cdot 30^{s+5} \cdot(s+2)^{4.5}\cdot k^2 (1+\log k)(1+\log (1.4n))\cdot (k\log p_s)^s\cdot 1\cdot 4.5k^2\log k\nonumber\\
	&= -1.4\cdot 30^{s}\cdot 30^5 \cdot s^{4.5}\cdot \left(1+\dfrac{2}{s}\right)^{4.5}\cdot k^2\log k \left(1+\dfrac{1}{\log k}\right)\cdot \log n\left(\dfrac{1}{\log n}+\dfrac{\log 1.4}{\log n}+1\right)\nonumber\\
	&\qquad\cdot k^s(\log p_s)^s\cdot 4.5k^2\log k\nonumber\\
	&> -1.7\cdot 10^{10} \cdot 30^s s^{4.5}k^{4+s}(\log k)^2(\log p_s)^s \log n.
\end{align}
In the above computation, we have used the fact that $n \geq 4$, $k\ge 2$ and $s\ge 2$. Comparing \eqref{eq3.2p} and \eqref{eq3.4p}, we get
\begin{align*}
	n\log \alpha-\log 1.82 &< 1.7\cdot 10^{10} \cdot 30^s s^{4.5}k^{4+s}(\log k)^2(\log p_s)^s \log n,
\end{align*}
which leads to
\begin{align}\label{eq3.5p}
	n&<4\cdot 10^{10} \cdot 30^s s^{4.5}k^{4+s}(\log k)^2(\log p_s)^s \log n.
\end{align}
We also make use of the bound $p_m < m^2$ valid for all $m \ge 2$, which is a consequence of Corollary to Theorem~3, found on page~69 of \cite{Ros}. Incorporating this estimate, inequality~\eqref{eq3.5p} transforms into
\begin{align*}
	\dfrac{n}{\log n}&<4\cdot 10^{10} \cdot  s^{4.5}k^{4+s}(\log k)^2(60\log s)^s  .
\end{align*}
By Lemma \ref{Guzp} with $x:=n$, $m:=1$ and $T:=4\cdot 10^{10} \cdot  s^{4.5}k^{4+s}(\log k)^2(60\log s)^s >(4m^2)^m=4$, we have 
\begin{align*}
	n&<2^1\cdot 4\cdot 10^{10} \cdot  s^{4.5}k^{4+s}(\log k)^2(60\log s)^s \left(\log (4\cdot 10^{10} \cdot  s^{4.5}k^{4+s}(\log k)^2(60\log s)^s)\right)^1\nonumber\\
	&< 8\cdot 10^{10}  s^{4.5}(\log k)^2 k^{4+s}(60\log s)^s \left(\log (4\cdot 10^{10})+4.5\log  s+(4+s)\log k+2\log\log k+s\log(60\log s)\right)\nonumber\\
	&< 8\cdot 10^{10}  s^{4.5}k^{6+s}(60\log s)^s \cdot ks\left(\dfrac{25}{ks}+\dfrac{4.5}{ks}\log  s+\left(\dfrac{1}{k}+\dfrac{4}{ks}\right)\log k+\frac{2\log\log k}{ks}+\dfrac{\log(60\log s)}{k}\right)\nonumber\\
	&< 8\cdot 10^{10}  s^{4.5}k^{6+s}(60\log s)^s \cdot ks\left(\dfrac{13}{s}+\dfrac{2.5}{s}\log  s+\left(0.5+\dfrac{2}{s}\right)\log k+\frac{0.2}{s}+\dfrac{\log(60\log s)}{2}\right)\nonumber\\
	&<8\cdot 10^{10}  s^{5.5}k^{7+s}(60\log s)^s (10s+2k),
\end{align*}
where we have used the fact that $\log k<k$, $\log s<s$ and $(0.5+2/s)\log k<2k$ for $k\ge 2$, $s\ge 2$ and 
$$13/s+(2.5/s)\log s+(\log(60\log s))/s+0.2/s<10s\qquad {\text{\rm for all}}\qquad s\ge 2.
$$
Therefore,
\begin{align*}
	n	&<8\cdot 10^{10} s^{5.5}k^{7+s}(60\log s)^s (10s+2k),
\end{align*}
and hence
\begin{align*}
	\log n	&<\log (8\cdot 10^{10})+5.5\log s +(7+s)\log k +s\log(60\log s)+\log (10s+2k)\\
	&<25.2+5.5\log s +s\log(60\log s)+5s\log k+\log (10s+2k), ~~\text{since}~ 7+s<5s ~~\text{for all}~s\ge 2,\\
	&=s \log s \left(\dfrac{25.2}{s\log s}+\dfrac{5.5}{s}+\dfrac{\log(60\log s)}{\log s}\right)+ 5s \log k +  \log(10s+2k),\\
	&<27s \log s + 5s \log k +  \log(10s+2k).
\end{align*}
This completes the proof of Lemma \ref{lem3.1p}.
\end{proof}
To move forward, note that whenever $k \le s$, we can invoke Lemma \ref{lem3.1p}, which yields
\begin{align*}
	\log n	&<27s \log s + 5s \log s +  \log(12s)<37s\log s,
\end{align*}
for $s\ge 2$. By the classical estimate $p_s > s\log s$, stated as relation (3.12) on page 69 of \cite{Ros}, it follows that
\begin{align}\label{ps0}
p_s>s\log s>(1/37)\log n.
\end{align}
Henceforth, we proceed under the assumption that $s < k$. Under this condition, the assertion of Lemma \ref{lem3.1p} takes the form
\begin{align}\label{eq3.6p}
	\log n	&<32s \log k  +  \log(12k)<37s\log k,
\end{align}
for $s\ge 2$ and $k\ge 2$. We proceed by distinguishing between two cases.

\subsection{The case $n\ge \varphi^{k/2}$}

Here, we have that $(k/2)\log \varphi \le \log n$,
so that 
\begin{align*}
k \le \dfrac{2}{\log \varphi}\log n<\dfrac{2}{\log \varphi}\cdot 37s\log k<154s\log k.	
\end{align*}
Therefore, $k/\log k < 154s$, from which we apply Lemma \ref{Guzp} and get
\begin{align}\label{eq3.7p}
	k<2547s\log s\qquad\text{and hence}\qquad \log k < 14\log s,
\end{align}
both of which hold for $s\ge 2$. 

Finally, we again use Lemma \ref{lem3.1p} together with the relations in \eqref{eq3.7p} to conclude that 
\begin{align*}
	\log n&<  27s \log s + 5s \cdot 14\log s +  \log(10s+2\cdot 2547s\log s)<104s\log s,
\end{align*}
for $s\ge 2$. As a result, we get 
\begin{align}\label{ps1}
p_s > s \log s > (1/104) \log n	.
\end{align}

\subsection{The case $n< \varphi^{k/2}$}\label{subsec3.3p}
Assume we have a real number $\lambda > 0$ such that $\alpha + \lambda = \varphi^2$. Since $\varphi^2(1 - \varphi^{-k} ) < \alpha < \varphi^2$, then 
$$\lambda =\varphi^2-\alpha < \varphi^2 - \varphi^2(1 - \varphi^{-k} ) = 1/\varphi^{k-2}.$$
That is, $\lambda \in (0, 1/\varphi^{k-2})$. On the other hand,	
	\begin{align*}
		\alpha^{n} &= (\varphi^2 - \lambda)^{n} = \varphi^{2n} \left(1 - \frac{\lambda}{\varphi^2}\right)^{n} 
		= \varphi^{2n}\left(e^{\log(1-\lambda/\varphi^2)}\right)^{n}\\ 
		&\ge \varphi^{2n}e^{-2(\lambda/\varphi^2) n} > \varphi^{2n}e^{-\lambda n}
		\ge  \varphi^{2n}(1-\lambda n) ,
	\end{align*}
where we used the fact that $\log(1 - x) > -2x$ for all $x < 1/2$ and $e^{-x} \ge 1 - x$ for all $x \in \mathbb{R}$. In fact, for us, we have $x:=\lambda/\varphi^2<(1/\varphi^{k-2})/\varphi^2=1/\varphi^k<0.5$ for all $k\ge 2$. 
	
Furthermore, 
$$\lambda n < \dfrac{n }{\varphi^{k-2} }< \dfrac{\varphi^{k/2}}{\varphi^{k-2}} = \varphi^{2-k/2},$$
implying that $\alpha^{n} > \varphi^{2n}(1-\lambda n)>\varphi^{2n}(1-\varphi^{2-k/2})$.
Again, since $\varphi^2(1 - \varphi^{-k} ) < \alpha < \varphi^2$, it follows that
	\[
	\varphi^{2n}-\dfrac{\varphi^{2n}}{\varphi^{k/2-2}}=\varphi^{2n}(1-\varphi^{2-k/2}) < \alpha^{n} < \varphi^{2n} <\varphi^{2n}+ \dfrac{\varphi^{2n}}{\varphi^{k/2-2}}.
	\]
Therefore
	\begin{equation} \label{eq3.9p}
		\left| \alpha^{n} -\varphi^{2n}\right| < \dfrac{\varphi^{2n}}{\varphi^{k/2-2}}.
	\end{equation}
We now turn our attention to the function $f_k(x)$ defined in~\eqref{eq:fk}. By applying the Mean Value Theorem, there exists a point $\omega \in (\alpha, \varphi^2)$ such that 
\[
f_k(\alpha) = f_k(\varphi^2) + (\alpha - \varphi^2) f_k'(\omega).
\]
It is worth noting that for all $k \geq 2$, we obtain
\begin{align*}
	|f_k'(\omega)| &= \dfrac{(k+1)(2\omega-\omega^2-1)-k}{((k + 1)\omega^2 - 3k\omega + k - 1)^2 } 
	< \dfrac{2\omega(k+1)}{((k + 1)\omega^2 - 3k\omega + k - 1 )^2}\\
	&< \dfrac{5.3(k+1)}{(2.5921(k + 1) - 7.86k + k - 1 )^2},\quad\text{since}\quad\omega \in (\alpha, \varphi^2)\subset (\varphi^2(1 - \varphi^{-k} ),\varphi^2)\subset(1.61,2.62),\\
	&<k.
\end{align*}
In fact, the last inequality follows directly because the function $5.3(k+1)/(2.5921(k + 1) - 7.86k + k - 1 )^2$ is at most 0.5 and is reducing for all $k\ge 2$. Hence,
	\begin{equation} \label{eq3.10p}
		\left|f_k(\alpha) - f_k(\varphi^2)\right| =\left|\alpha - \varphi^2\right||f'_k(\omega)| = \lambda |f'_k(\omega)| < k\lambda<\dfrac{k}{\varphi^{k-2}}.
	\end{equation}

From the above, if we write 
	\[
	\alpha^{n} = \varphi^{2n} + \delta \qquad \text{and} \qquad f_k(\alpha) = f_k(\varphi^2) + \eta,
	\]
then inequalities \eqref{eq3.9p} and \eqref{eq3.10p} become
	\begin{equation} \label{eq3.12p}
		|\delta| < \dfrac{\varphi^{2n}}{\varphi^{k/2-2}} \qquad \text{and} \qquad |\eta|<\dfrac{k}{\varphi^{k-2}}.
	\end{equation}
Moreover, since $f_k(\varphi^2) = (5-\sqrt5)/10$ for all $k \geq 2$, we have
	\begin{align} \label{eq3.13p}
		f_k(\alpha) \alpha^{n} &=\left( f_k(\varphi^2) + \eta\right) (\varphi^{2n} + \delta )=\left( \dfrac{5-\sqrt{5}}{10} + \eta\right) (\varphi^{2n} + \delta )\nonumber\\
		& = \left( \dfrac{5-\sqrt{5}}{10}\right)\varphi^{2n} + \left( \dfrac{5-\sqrt{5}}{10}\right)\delta + \varphi^{2n}\eta + \eta \delta.
	\end{align}
Therefore, using \eqref{eq2.6m} and relations \eqref{eq3.12p} and \eqref{eq3.13p}, we get
	\begin{align*}
		\left|p_1^{\beta_1} \cdots p_s^{\beta_s} -\left( \dfrac{5-\sqrt{5}}{10}\right)\varphi^{2n}\right| &= \left| \left(P^{(k)}_n - f_k(\alpha)\alpha^{n}\right) +\left( \left( \dfrac{5-\sqrt{5}}{10}\right)\delta + \varphi^{2n}\eta + \eta \delta\right) \right| \\
		&< \frac{1}{2}+ \left( \dfrac{5-\sqrt{5}}{10}\right)\dfrac{\varphi^{2n}}{\varphi^{k/2-2}}+ \varphi^{2n}\cdot\dfrac{k}{\varphi^{k-2}}+\dfrac{\varphi^{2n}}{\varphi^{k/2-2}}\cdot \dfrac{k}{\varphi^{k-2}}.
	\end{align*}
Dividing both sides by $\varphi^{2n}(5-\sqrt5)/10$, we get	
	\begin{align*} 
\left|p_1^{\beta_1} \cdots p_s^{\beta_s} \left( \dfrac{10}{5-\sqrt{5}}\right)\varphi^{-2n}-1\right|
&<\frac{10}{2(5-\sqrt5)\varphi^{2n}}+ \dfrac{1}{\varphi^{k/2-2}}+ \dfrac{10k}{\varphi^{k-2}(5-\sqrt5)}+\dfrac{10k}{\varphi^{3k/2-4}(5-\sqrt5)}\nonumber\\
	&<\frac{2}{\varphi^{2n}}+ \dfrac{1}{\varphi^{k/2-2}}+ \dfrac{4k}{\varphi^{k-2}}+\dfrac{4k}{\varphi^{3k/2-4}}\nonumber\\
	&< \frac{2}{\varphi^{2n}}+ \dfrac{3}{\varphi^{k/2}}+ \dfrac{17}{\varphi^{k/2}}+\dfrac{22}{\varphi^{k/2}}\nonumber\\
	&\le \frac{44}{\varphi^{\min\{2n,k/2\}}} .
\end{align*}
Therefore, we have
	\begin{align} \label{eq3.14p}
|\Gamma_2|:=\left|p_1^{\beta_1} \cdots p_s^{\beta_s} \left( \dfrac{10}{5-\sqrt{5}}\right)\varphi^{-2n}-1\right|< \frac{44}{\varphi^{\min\{2n,k/2\}}}.
	\end{align} 
Now, we intend to apply Theorem \ref{thm:Matp} on the left hand side of \eqref{eq3.14p}. It should be noted that $\Gamma_2\ne 0$, otherwise we would have
$$p_1^{\beta_1} \cdots p_s^{\beta_s} =\left( \dfrac{5-\sqrt{5}}{10}\right)\varphi^{2n},$$ which is not true because the left hand side is rational while the right hand side is irrational. Here, $t := s+2$,
\begin{equation}\nonumber
	\begin{aligned}
	\gamma_{i}&:=p_i~~ \text{for}~i=1,2,\ldots, s, ~~~\gamma_{s+1}:=\varphi,~~~~~\gamma_{s+2}:=10/(5-\sqrt5),\\
		b_{i}&:=\beta_i~~ \text{for}~i=1,2,\ldots, s, ~~~b_{s+1}:=-2n, ~~b_{s+2}:=1.
	\end{aligned}
\end{equation}
The algebraic number field containing the above $\gamma_i$'s is $\mathbb{K} := \mathbb{Q}(\varphi)$, so $D = 2$. Since $h(\gamma_{i})=\log p_i \le \log p_s$ for all $i=1,2,\ldots, s$, we take $A_i:=2 \log p_s$ for all $i=1,2,\ldots, s$.  Additionally, 
$$h(\gamma_{s+1})=h(\varphi) = (\log \varphi)/2,$$
so we take $A_{s+1}:=\log \varphi$. Lastly, 
\[
h(\gamma_{s+2}) = h\left(10/(5-\sqrt5)\right) \le h(10)+h(5)+h(\sqrt5)+\log 2< 4\log 10,
\] 
so, we take $A_{s+2}:=8\log 10$.

Next, $B \geq \max\{|b_i|:i=1,2,\ldots, s,\ldots, s+2\}$. Notice that $b_i=\beta_i<1.4n$, for all $i=1,2,\ldots, s$ by relation  \eqref{eq2.7p}, so we take $B:=2n$. Now, by Theorem \ref{thm:Matp},
\begin{align}\label{eq3.15p}
\log |\Gamma_2| &> -1.4\cdot 30^{s+5} \cdot(s+2)^{4.5}\cdot 2^2 (1+\log 2)(1+\log (2n))\cdot (2\log p_s)^s\cdot \log\varphi\cdot 8\log 10\nonumber\\
&> -1.4\cdot 30^{s}\cdot 30^5 \cdot s^{4.5}\cdot \left(1+\dfrac{2}{s}\right)^{4.5}\cdot 2^2 (1+\log 2)\cdot \log n\left(\dfrac{1}{\log n}+\dfrac{\log 2}{\log n}+1\right)\nonumber\\
&\qquad\cdot 2^s(\log p_s)^s\cdot \log\varphi\cdot 8\log 10\nonumber\\
&> -1.1\cdot 10^{11} \cdot 60^s s^{4.5}(\log p_s)^s \log n.
\end{align}
Comparing \eqref{eq3.14p} and \eqref{eq3.15p}, we get
\begin{align*}
	\min\{2n,k/2\}\log \varphi-\log 44 &<1.1\cdot 10^{11} \cdot  s^{4.5}\log n\cdot(120\log s)^s,
\end{align*}
since $p_s<s^2$ for all $s\ge 2$. 

Therefore
\begin{align}\label{min}
	\min\{2n,k/2\} < 2.3\cdot 10^{11} \cdot  s^{4.5}\log n\cdot(120\log s)^s.
\end{align}

We proceed by distinguishing between two cases from \eqref{min}.

\textbf{Case I:} Suppose $\min\{2n,k/2\}:=k/2$, then $k/2 < 2.3\cdot 10^{11} \cdot  s^{4.5}\log n\cdot(120\log s)^s$, or equivalently 
\begin{align*}
k &< 4.6\cdot 10^{11} \cdot  s^{4.5}\log n\cdot(120\log s)^s\\
&< 4.6\cdot 10^{11} \cdot  s^{4.5}(37s\log k)\cdot(120\log s)^s
\end{align*}
since $\log n<37s\log k$ in \eqref{eq3.6p}. Therefore, we can write
\begin{align*}
	\dfrac{k}{\log k}<2 \cdot 10^{13}\cdot s^{5.5}\cdot(120\log s)^s.
\end{align*}
We again apply Lemma \ref{Guzp} with the data: $x:=k$, $m:=1$ and $T:=2 \cdot 10^{13}\cdot s^{5.5}\cdot(120\log s)^s>(4m^2)^m=4$, for $s\ge 2$. We get 
\begin{align*}
	k&<2^1\cdot 2 \cdot 10^{13}\cdot s^{5.5}\cdot(120\log s)^s (\log (2 \cdot 10^{13}\cdot s^{5.5}\cdot(120\log s)^s))^1\nonumber\\
	&= 4\cdot 10^{13}  s^{5.5}(120\log s)^s \left(\log (2\cdot 10^{13})+5.5\log  s+s\log(120\log s)\right)\nonumber\\
	&<4\cdot 10^{13}  s^{5.5}(120\log s)^s\cdot s\log s \left(\dfrac{31}{s\log s}+\dfrac{5.5}{s}+\dfrac{\log 120}{\log s}+\dfrac{\log\log s}{\log s}\right)\nonumber\\
	&<1.3\cdot 10^{15}s^{6.5}(120\log s)^s\log s.
\end{align*}
Therefore, 
\begin{align}\label{eq3.16p}
	\log k &<\log (1.3\cdot 10^{15})+6.5\log s+s\log(120\log s)+\log\log s\nonumber\\
	&< 35+6.5\log s+s\log120+s\log\log s+\log\log s\nonumber\\
	&=s\log s\left(\dfrac{35}{s\log s}+\dfrac{6.5}{s}+\dfrac{\log 120}{\log s}+\dfrac{\log \log s}{\log s}+\dfrac{\log \log s}{s\log s}\right)\nonumber\\
	&<35s\log s.
\end{align}
Recall that we are in the subsection when $n < \varphi^{k/2}$ , therefore $\log n <(k/2) \log \varphi < k$. This and relation \eqref{eq3.16p} tell us that $\log \log n < \log k < 35 s \log s$, hence
\begin{align}\label{ps2}
p_s > s \log s > (1/35) \log \log n.
\end{align}

\textbf{Case II:} If $\min\{2n,k/2\}:=2n$, then $2n < 2.3\cdot 10^{11} \cdot  s^{4.5}\log n\cdot(120\log s)^s$, or equivalently 
\begin{align*}
	\dfrac{n}{\log n} &< 1.2\cdot 10^{11} \cdot  s^{4.5}\cdot(120\log s)^s.
\end{align*}
We again apply Lemma \ref{Guzp} with the data: $x:=n$, $m:=1$ and $T:=1.2\cdot 10^{11} \cdot  s^{4.5}\cdot(120\log s)^s>(4m^2)^m=4$, for $s\ge 2$. We get 
\begin{align*}
	n&<2^1\cdot 1.2\cdot 10^{11} \cdot  s^{4.5}\cdot(120\log s)^s (\log (1.2\cdot 10^{11} \cdot  s^{4.5}\cdot(120\log s)^s)^1\nonumber\\
	&= 2.4\cdot 10^{11}  s^{4.5}(120\log s)^s \left(\log (1.2\cdot 10^{11})+4.5\log  s+s\log(120\log s)\right)\nonumber\\
	&<2.4\cdot 10^{11}  s^{4.5}(120\log s)^s\cdot s\log s \left(\dfrac{26}{s\log s}+\dfrac{4.5}{s}+\dfrac{\log 120}{\log s}+\dfrac{\log\log s}{\log s}\right)\nonumber\\
	&<6.6\cdot 10^{12}s^{5.5}(120\log s)^s\log s.
\end{align*}
Finally,
\begin{align*}
		\log n &<\log (6.6\cdot 10^{12})+5.5\log s+s\log(120\log s)+\log\log s\nonumber\\
	&< 30+5.5\log s+s\log 120+s\log\log s+\log\log s\nonumber\\
	&=s\log s\left(\dfrac{30}{s\log s}+\dfrac{5.5}{s}+\dfrac{\log 120}{\log s}+\dfrac{\log \log s}{\log s}+\dfrac{\log \log s}{s\log s}\right)\nonumber\\
	&<31s\log s.
\end{align*}
Therefore, 
\begin{align}\label{ps3}
	p_s > s \log s > (1/31) \log  n.
\end{align}
Comparing the four relations \eqref{ps0}, \eqref{ps1}, \eqref{ps2} and \eqref{ps3} completes the proof of Theorem \ref{1.1p}. \qed

\section{Proof of Theorem \ref{1.2p}}\label{Sec4}

Here, we are interested in solving the Diophantine equation \eqref{eq1.1p}
in non-negative integers $n$, $k$, $a$, $b$, $c$, $d$, $e$ with $k \ge 2$ and $n \ge 4$. We start by distinguishing between the cases $n\le k+1$ and $n>k+1$.
\subsection{The case $4\le n\le k+1$}
In this case, $P_n^{(k)} = F_{2n-1}$, so we turn out attention to solving 
\begin{align}\label{fm}
	 F_{m} = 2^a \cdot 3^b\cdot 5^c\cdot 7^d,
\end{align}
where $m:=2n-1$. In fact, $m=2n-1\le 2(k+1)-1=2k+1$ and $m=2n-1\ge 2\cdot 4-1=7$. Therefore,
$$7\le m\le 2k+1.$$

For an integer $\ell\ge 2$, the Diophantine equation $F_{m}^{(\ell)} = 2^a \cdot 3^b\cdot 5^c\cdot 7^d$ was completely solved in \cite{Brl}, for all integers $m>\ell+1$ and it was found out that for $\ell=2$, the solutions are
\begin{align*}
	F_4=3,\qquad F_5 = 5, \qquad F_6 = 8, \qquad F_8 = 21\qquad\text{and}\qquad F_{12}=144,
\end{align*}
and none of these are $P_n^{(k)}$ values with $n\ge 4$. 

On the other hand, if $m\le \ell+1$, then $F_{m}^{(\ell)}=2^{m-2}$. It is known that $1$, $2$, $8$ are the only powers of two that appear in our familiar Fibonacci sequence, (when $\ell :=2$). One proof of this fact follows from Carmichael's Primitive Divisor theorem, which states that for $n$ greater than $12$, the $m^{\text{th}}$ Fibonacci number $F_m$ has at least one prime factor that is not a factor of any previous Fibonacci number. Since $m\ge 7$, none of these powers 2 make it to our required solutions to \eqref{fm}.

From now on, $n>k+1$.
\subsection{The case $n>k+1$}
We proceed in a similar way as in Section \ref{Sec3p}. Specifically, we prove following estimates.
\begin{lemma}\label{lem4.1p}
If $\mathcal{P}(P_n^{(k)} ) \le 7$, then:
\begin{enumerate}[(a)]
	\item The inequality	
	\begin{align*}
		n  < 2.3\cdot 10^{23} k^7 (\log k)^3,
	\end{align*} 	
	holds for $k \ge 2$ and $n \ge 4$.
	\item If $k > 2500$, then 
	\begin{align*}
		k < 1.2\cdot 10^{28}\qquad	\text{and}\qquad	n <2.5\cdot 10^{225}.
	\end{align*}
\end{enumerate}
\end{lemma}
\begin{proof}~
\begin{enumerate}[(a)]
	\item To establish the first part, we follow a similar line of reasoning as in Subsection~\ref{subsec3.1p}. Specifically, when $s = 4$, inequality~\eqref{eq3.2p} gives
	\begin{align}\label{eq4.1p}
		\left|2^{a}\cdot 3^b\cdot 5^c\cdot 7^d\cdot \alpha^{-n}\cdot(f_k(\alpha))^{-1}-1\right|
		&<\dfrac{1.82}{\alpha^{n}},
	\end{align}
for which we obtain as before, by substituting $s=4$ in \eqref{eq3.5p}, that 	
	\begin{align*}
		n&<4\cdot 10^{10} \cdot 30^4 \cdot 4^{4.5}k^{3+4}(\log k)^2(\log 7)^4 \log n\\
		&<1.5\cdot 10^{21}k^7(\log k)^2 \log n.
	\end{align*}
Specifically, we have 
	\begin{align*}
	\dfrac{n}{\log n}
	&<1.5\cdot 10^{21}k^7(\log k)^2 .
\end{align*}
We now apply Lemma \ref{Guzp} with  $x:=n$, $m:=1$, $T:=1.5\cdot 10^{21}k^7(\log k)^2 >(4m^2)^m=4$ for $k\ge 2$. We get 
\begin{align*}
	n&<2\cdot 1.5\cdot 10^{21}k^7(\log k)^2 \log (1.5\cdot 10^{21}k^7(\log k)^2)\\
	&=3\cdot 10^{21}k^7(\log k)^2 \left(\log (1.5\cdot 10^{21})+7\log k+2\log\log k\right)\\
	&<3\cdot 10^{21}k^7(\log k)^3 \left(\dfrac{49}{\log k}+7+\dfrac{2\log\log k}{\log k}\right)\\
	&<2.3\cdot 10^{23}k^7(\log k)^3.
\end{align*}
\item In the second part, if $k>2500$, then 
\begin{align*}
	n	<2.3\cdot 10^{23}k^7(\log k)^3<\varphi^{k/2}.
\end{align*}
For this reason, we can use the same arguments from Subsection \ref{subsec3.3p} relation \eqref{eq3.14p} to write 
	\begin{align} \label{eq4.2p}
	\left|2^{a}\cdot 3^{b}\cdot 5^c\cdot 7^d\cdot \left( \dfrac{10}{5-\sqrt{5}}\right)\varphi^{-2n} -1\right|< \frac{44}{\varphi^{\min\{2n,k/2\}}},
\end{align}
from which after applying Matveev's result with $s=4$ as in Subsection \ref{subsec3.3p} we get
\begin{align*}
\min\{2n,k/2\} < 2.3\cdot 10^{11} \cdot  4^{4.5}\log n\cdot(120\log 4)^4<9.1\cdot 10^{22}\log n.
\end{align*}

In this subsection, $n>k+1$, so $2n>n>k+1>k/2$. This means that $\min\{2n,k/2\}:=k/2$. Therefore, we are in Case I of Subsection \ref{subsec3.3p} and for $s=4$, we have
\begin{align*}
	k<1.3\cdot 10^{15}\cdot 4^{6.5}(120\log 4)^4\log 4 < 1.2\cdot 10^{28}.
\end{align*}
Substituting this upper bound on $k$ in part (a) of Lemma \ref{lem4.1p}, we get
\begin{align*}
	n  &< 2.3 \cdot 10^{23}(1.2\cdot 10^{28})^7 (\log 1.2\cdot 10^{28})^3<2.5\cdot 10^{225}.
\end{align*}
\end{enumerate}
This completes the proof of Lemma \ref{lem4.1p}.
\end{proof}
To complete the proof of Theorem \ref{1.2p}, we proceed in two cases, that is, the case $k\le2500$ and the case $k>2500$.

\subsubsection{The case $k\le 2500$.}
In this subsection, we handle the values of $k$ in $[2,2500]$. From Lemma \ref{lem4.1p}, we know that if $k \le 2500$, then $n < 6.8 \cdot 10^{49}$. Our goal now is to make this upper limit for $n$ smaller. For that purpose, we define
\[
\tau_1 := a \log 2 + b \log 3 + c \log 5 + d \log 7 - n \log \alpha - \log f_k(\alpha),
\]
so that \eqref{eq4.1p} can be rewritten as
\begin{equation}\label{eq4.3p}
	|e^{\tau_1} - 1| < \frac{1.82}{\alpha^{n}}.
\end{equation}
Note that $\tau_1 \ne 0$, and this can be shown in the same way we proved that $\Gamma_1 \ne 0$. Now, if $\tau_1 > 0$, then it follows that $e^{\tau_1} - 1 > 0$. Using inequality~\eqref{eq4.3p}, we get
$
0 < \tau_1 < 1.82/\alpha^{n},
$
where we applied the basic inequality $x \le e^x - 1$ for all real numbers $x$. 

Let us now look at the case when $\tau_1 < 0$. Since $n\ge 4$, then we know that $1.82/\alpha^{n} < 1/2$ for every $k \ge 2$. From~\eqref{eq4.3p}, this implies that $|e^{\tau_1} - 1| < 1/2$, meaning that $e^{\tau_1} < 1.5$. Since $\tau_1$ is negative, we use this to get
\begin{equation*}
0 < |\tau_1| \le e^{|\tau_1|} - 1 = e^{|\tau_1|} \cdot |e^{\tau_1} - 1| < \frac{4}{\alpha^{n}}.
\end{equation*}
So in both situations, we conclude that
\begin{equation*}
	|\tau_1| < \frac{4}{\alpha^{n}},
\end{equation*}
for all $k \ge 2$ and $n \ge 4$.
For each $k \in [2, 2500]$, we used the LLL-algorithm to compute a lower bound for the smallest nonzero number of the form $|\tau_1|$, with integer coefficients not exceeding $2.3 \cdot 10^{23} k^7 (\log k)^3$ in absolute value. Specifically, we consider the approximation lattice
$$ \mathcal{A}=\begin{pmatrix}
	1 & 0 & 0 & 0 & 0 & 0 \\
	0 & 1 & 0 & 0 & 0 & 0 \\
	0 & 0 & 1 & 0 & 0 & 0 \\
	0 & 0 & 0 & 1 & 0 & 0 \\
	0 & 0 & 0 & 0 & 1 & 0 \\
	0 & 0 & 0 & 0 & 0 & 0 \\
	\lfloor C\log 2\rfloor & \lfloor C\log 3\rfloor& \lfloor C\log 5 \rfloor& \lfloor C\log 7 \rfloor & \lfloor C\log\alpha \rfloor & \lfloor C\log f_k(\alpha) \rfloor
\end{pmatrix} ,$$
with $C:= 10^{299}$ and choose $y:=\left(0,0,0,0,0,0,0\right)$. 

Now, by Lemma \ref{lem2.5p}, we get $$|\tau_1|>c_1=10^{-53}\qquad\text{and}\qquad c_2=5.8\cdot10^{51}.$$
So, Lemma \ref{lem2.6p} gives $S=2.312\cdot 10^{100}$ and $T=2.05\cdot 10^{50}$. Since $c_2^2\ge T^2+S$, then choosing $c_3:=4$ and $c_4:=\log\alpha$, we get $n\le 1214$.

To finish, we created a basic Python script (see Appendix \ref{app1}) to examine the $k$-generalized Pell numbers for values of $k\in [2, 2500]$ and  $n\in [k + 2, 1214]$. Instead of fully factoring the numbers, we repeatedly divided each one by $2$, $3$, $5$, and $7$ as long as possible. This allowed us to check whether each term is composed only of these four primes. The values that satisfied this condition are exactly those listed in Theorem \ref{1.2p}, and the Python script ran for about 16 hours. This completes our investigation for $k\in [2, 2500]$.

\subsubsection{The case $k>2500$}
Finally, we consider the case when $k > 2500$. At this stage, we aim to lower the upper limit on $k$ obtained in Lemma \ref{lem4.1p}. To achieve this, let us define
\[
\tau_2 := a \log 2 + b \log 3 + c \log 5 + d \log 7 + \log\left( \dfrac{10}{5-\sqrt{5}}\right)  -2n\log \varphi,
\]
so that inequality \eqref{eq4.2p} can be rewritten as
\begin{equation*}
	|e^{\tau_2} - 1| <\frac{44}{\varphi^{\min\{2n,k/2\}}}.
\end{equation*}
Again, $\tau_2$ is clearly nonzero by a similar argument used to deduce that $\Gamma_2\ne 0$. As before, regardless of the sign of $\tau_2$, we arrive at the inequality
\begin{equation}
	|\tau_2| < \frac{66}{\varphi^{\min\{2n,k/2\}}}.
\end{equation}

So, we consider the approximation lattice
$$ \mathcal{A}=\begin{pmatrix}
	1 & 0 & 0 & 0 & 0 & 0 \\
	0 & 1 & 0 & 0 & 0 & 0 \\
	0 & 0 & 1 & 0 & 0 & 0 \\
	0 & 0 & 0 & 1 & 0 & 0 \\
	0 & 0 & 0 & 0 & 1 & 0 \\
	0 & 0 & 0 & 0 & 0 & 0 \\
	\lfloor C\log 2\rfloor & \lfloor C\log 3\rfloor& \lfloor C\log 5 \rfloor& \lfloor C\log 7 \rfloor & \lfloor C\log 10/(5-\sqrt{5}) \rfloor & \lfloor C\log \varphi \rfloor
\end{pmatrix} ,$$
with $C:= 10^{1356}$ and choose $y:=\left(0,0,0,0,0,0\right)$. By Lemma \ref{lem2.5p}, we get $$|\tau_2|>c_1=10^{-228}\qquad\text{and}\qquad c_2=2.7\cdot10^{226}.$$
So, Lemma \ref{lem2.6p} gives $S=3.125\cdot 10^{451}$ and $T=7.51\cdot 10^{225}$. Since $c_2^2\ge T^2+S$, then choosing $c_3:=66$ and $c_4:=\log \varphi$, we get $\min\{2n,k/2\} \le 5414$. Since $2n>n>k+1>k/2$, then $\min\{2n,k/2\}:=k/2$ implying that $k\le 10828$. Thus, the first part of Lemma \ref{lem4.1p} tells us that $n < 3.3 \cdot 10^{54}$.

With this new upper bound for $n$ we repeat the LLL-algorithm to get a lower bound of $|\tau_2|$. With the same approximation lattice and $C:=1.3\cdot 10^{327}$, we get $|\tau_2|>c_1=10^{-57}$, $c_2=3.6\cdot10^{56}$, $S=5.5\cdot 10^{109}$ and $T=9.91\cdot 10^{54}$. We then obtain that $k \leq 2606$ and the first part of Lemma \ref{lem4.1p} tells us that $n < 9.2\cdot 10^{49}$.

Again, using this improved upper limit for $n$, we re-apply the LLL-algorithm to derive a lower bound for $|\tau_2|$ than before. By using the same lattice for the approximation and setting $C := 10^{300}$, we find that $|\tau_2| > c_1 = 10^{-52}$,  $c_2 = 1.4 \cdot 10^{51}$, $S = 4.3 \cdot 10^{100}$ and $T = 2.77 \cdot 10^{50}$. From this, it follows that $k \leq 2398$, contradicting our earlier assumption that $k>2500$. Therefore, Theorem \ref{1.2p} is proved.
 \qed

\section*{Acknowledgments} 
The author was funded by the cost centre 0730 of the Mathematics Division at Stellenbosch University.

\section*{Addresses}
$ ^{1} $ Mathematics Division, Stellenbosch University, Stellenbosch, South Africa.

 Email: \url{hbatte91@gmail.com}

\pagebreak
\appendix
\section{Appendices}

\subsection{Py Code I}\label{app1}
\begin{verbatim}
# Dictionary to memoize computed values
memo = {}

def k_generalized_pell_iterative(n, k):
    # Handle trivial cases
    if n < 2 - k:
        return 0
    elif n == 1:
        return 1

    # Check if already computed
    if (n, k) in memo:
        return memo[(n, k)]

    # Initialize the list of Pell numbers
    pell_nums = [0] * (2 - k) + [0, 1] + [0] * (n - 1)

    # Compute iteratively
    for i in range(2, n + 1):
        pell_nums[i] = 2 * pell_nums[i - 1] + sum(
            pell_nums[i - j] for j in range(2, k + 1) if i - j >= 0
        )
        memo[(i, k)] = pell_nums[i]

    return pell_nums[n]

def is_of_form_2a_3b_5c_7d(val):
    for prime in [2, 3, 5, 7]:
        while val % prime == 0 and val > 1:
            val //= prime
    return val == 1

# Search for valid Pell numbers with prime factors <= 7
for k in range(2, 2500):
    for n in range(k+2, 1214):
        pell_val = k_generalized_pell_iterative(n, k)
        if pell_val > 0 and is_of_form_2a_3b_5c_7d(pell_val):
             print(f"For k={k}, n={n}: P_n^({k}) = {pell_val}")
\end{verbatim}


\begin{thebibliography}{99}
	\bibitem{Bach}
	Bachabi, M., \& Togbe, A. (2024). On $k$-Pell numbers close to power of 2. \textit{Revista Colombiana de Matemáticas}, \textbf{58}(1), 67.
	
	\bibitem{Batte}
	Batte, H., \& Luca, F. (2024). On the largest prime factor of the $k$-generalized Lucas numbers. \textit{Boletín de la Sociedad Matemática Mexicana}, \textbf{30}(2), 33.
	
	\bibitem{BraHer}
	Bravo, J. J., \& Herrera, J. L. (2020). Repdigits in generalized Pell sequences. \textit{Arch. Math.}	\textbf{56}(4), 249--262.
	
	\bibitem{Herera}
	Bravo, J. J., Herrera, J. L., \& Luca, F.  (2021). On a generalization of the Pell sequence. \textit{Mathematica Bohemica}, 146(2), 199--213.
	
	\bibitem{Brl}
	Bravo, J. J., \& Luca, F. (2013). On the largest prime factor of the $k$-Fibonacci numbers, {\it International Journal of Number Theory\/} {\bf 9}, 1351--1366.
		
	\bibitem{Cas}
	Cassels, J. W. S. (2012). An introduction to the geometry of numbers. Springer Science \& Business Media.
	
	\bibitem{Weg}
	de Weger, B. M. (1987). Solving exponential Diophantine equations using lattice basis reduction algorithms, {\it Journal of Number Theory\/} {\bf 26},  325--367.
	
	
	\bibitem{GL}
	G{\'u}zman--Sanchez, S., \& Luca, F. (2014). Linear combinations of factorials and S-units in a binary recurrence sequence. {\it Annales Math{\'e}matiques du Qu{\'e}bec\/} {\bf  38}, 169--188.
	
	\bibitem{kilis}
	Kili\c c, E. (2013). On the usual Fibonacci and generalized order-$k$ Pell numbers. \textit{Ars Combin.} \textbf{109}, 391--403.
	
	\bibitem{LLL} 
	Lenstra, A. K., Lenstra, H. W., \& Lovász, L. (1982). Factoring polynomials with rational coefficients. {\it Mathematisches Annalen\/} {\bf 261}, 515--534.
	
	\bibitem{MAT}
	Matveev, E. M. (2000). An explicit lower bound for a homogeneous rational linear form in the logarithms of algebraic numbers. II. {\it Izvestiya: Mathematics\/} {\bf 64}, 1217.
			
	\bibitem{Ros}
	Rosser, J. B., \& Schoenfeld, L. (1962). Approximate formulas for some functions of prime numbers. {\it Illinois Journal of Mathematics} {\bf 6}, 64--94.
	
	\bibitem{SMA}
	Smart, N. P. (1998). The algorithmic resolution of Diophantine equations: a computational cookbook (Vol. 41). Cambridge University Press.
	
	
\end{thebibliography}
\end{document}